\documentclass[a4paper,reqno]{amsart}

\usepackage{amsmath}
\usepackage{amssymb}
\usepackage{graphicx}

\newtheorem{defn}{Definition}
\newtheorem{lem}[defn]{Lemma}
\newtheorem{thm}[defn]{Theorem}
\newtheorem{cor}[defn]{Corollary}
\newtheorem{prop}[defn]{Proposition}

\newcommand{\bU}{\bar U}
\newcommand{\bV}{\bar V}
\newcommand{\Ext}{\mathrm{Ext}}
\newcommand{\T}{\tilde}
\newcommand{\Z}{\mathbb Z}
\newcommand{\HGH}[2][G]{#2\backslash #1/#2}
\newcommand{\Hom}{\mathrm{Hom}}

\begin{document}
\begin{titlepage}
\title{Reduction of the Hall-Paige Conjecture to Sporadic Simple Groups}
\author{Stewart Wilcox}
\address{Stewart Wilcox \\ Department of Mathematics \\ Harvard University \\ 1 Oxford Street \\ Cambridge MA 02318 \\ USA}
\email{swilcox@fas.harvard.edu}
\end{titlepage}

\begin{abstract}
A complete mapping of a group $G$ is a permutation $\phi:G\rightarrow G$ such 
that $g\mapsto g\phi(g)$ is also a permutation. Complete mappings of $G$ are equivalent to 
tranversals of the Cayley table of $G$, considered as a latin square. In 1953, Hall and Paige 
proved that a finite group admits a complete mapping only if its Sylow-2 subgroup is trivial or 
non-cyclic. They conjectured that this condition is also sufficient. We 
prove that it is sufficient to check the conjecture for the 26 sporadic simple groups and the Tits group.
\end{abstract}
\maketitle
\section{Introduction}
All groups will be assumed finite. Let $G$ be a group. For the sake of brevity, we say $G$ 
is \emph{bad} if the Sylow-2 subgroup of $G$ is nontrivial and cyclic, and otherwise we say $G$ 
is \emph{good}. A \emph{complete mapping} of $G$ consists of an indexing set $I$ and bijections 
$a,b,c:I\rightarrow G$, such that
\[
  a(i)b(i)=c(i)
\]
for all $i\in I$. Note that $\phi=ba^{-1}$ and $\psi=ca^{-1}$ are bijections, so $G$ possesses a complete 
mapping if and only if there are permutations $\phi$ and $\psi$ of $G$ with $g\phi(g)=\psi(g)$. Complete 
mappings also have a combinatorial interpretation; a group possesses a complete mapping if and only if its 
Cayley table, which is a latin square, possesses an orthogonal mate \cite{Mann}.

Hall and Paige \cite{HP} proved that if $G$ possesses a complete mapping, then it is good; they also conjectured 
the converse (henceforth the ``HP conjecture"), and proved it in many special cases. They also proved many useful results, 
including Propositions \ref{solu}, \ref{HPcst} and \ref{ses} below.
\begin{prop}[\cite{HP} Theorem 6]\label{solu}
Any good soluble group possesses a complete mapping.
\end{prop}
In particular, any group of odd order is soluble \cite{FeitThompson} and therefore possesses a complete mapping (this 
can be easily shown directly by making $\phi$ the identity).

Recall that a left transversal of a subgroup $H\subseteq G$ is a set $X\subseteq G$ such that
\[
  G=\coprod_{x\in X}xH,
\]
and similarly for a right transversal.
\begin{prop}[\cite{HP} Theorem 1]\label{HPcst}
Suppose $H$ is a subgroup of $G$, and $H$ possesses a complete mapping. Suppose $X$ is both a left and right transversal 
for $H$, and $\phi$ and $\psi$ are permutations of $X$ such that
\[
  x\phi(x)H=\psi(x)H
\]
for $x\in X$. Then $G$ possesses a complete mapping.
\end{prop}
The following result is a direct corollary of Proposition \ref{HPcst}.
\begin{prop}[\cite{HP} Corollary 2]\label{ses}
Suppose $N$ is a normal subgroup of $G$ such that both $N$ and $H=G/N$ possess complete mappings. Then $G$ possesses a complete mapping.
\end{prop}
Recently many groups have been shown to satisfy the conjecture. The Mathieu groups $M_{11}$, $M_{12}$, $M_{22}$, 
$M_{23}$ and $M_{24}$, and some groups of Lie type, have been shown to possess complete mappings \cite{CM1,CM2,CM3,CM4}. Dalla Volta and 
Gavioli have shown that a minimal counterexample to the HP conjecture would have to be almost simple, or contain a central 
involution \cite{volta}. Continuing in this direction, we will show that a minimal counterexample must be one of the 26 sporadic simple 
groups or the Tits group. In a companion paper, Evans \cite{Tony07} deals with 26 of these groups (including an alternative treatment of 
the Mathieu groups), leaving the fourth Janko group as the only possible counterexample. John Bray reports that this group is also not a 
counterexample, thus completing the proof of the HP conjecture.

In Section \ref{sec:simple} we give two versions of Proposition \ref{ses} in which $N$ and $H$ are replaced by $\Z_2$ (see Propositions 
\ref{Z2ses} and \ref{sesZ2} respectively). These are used, along with Proposition \ref{solu}, to reduce the conjecture to simple groups. 

In Section \ref{sec:dbl} we prove a version of Proposition \ref{HPcst} in which the assumption on the cosets of $H$ is weakened 
(see Proposition \ref{lcst}). We also prove versions in which this assumption is replaced by an assumption on the double cosets of $H$, which 
is often easier to check in practice (see Corollaries \ref{dcstperm} and \ref{dcst}).

Finally in Section \ref{sec:lie}, we use the results of Section \ref{sec:dbl} and results about $(B,N)$-pairs to prove that a minimal 
counterexample cannot be a finite simple group of Lie type.

\section{Reduction to Simple Groups}\label{sec:simple}
We start with some well known results.
\begin{lem}\label{badind2}
Suppose $G$ is bad. Then there exists a characteristic subgroup of $G$ of index $2$.
\end{lem}
\begin{cor}\label{badcompl}
Suppose $G$ is bad. Then $G$ contains a characteristic subgroup $N$ of odd order, such that the quotient $G/N$ is a cyclic 2-group. In particular, $G$ is soluble.
\end{cor}
The first result follows by considering the inverse image of the alternating group under the regular 
representation $G\rightarrow S_G$. The second follows from the first by induction on $|G|$.

To prove the first version of Proposition \ref{ses}, we need the following well known 
combinatorial result, the proof of which is straightforward.
\begin{lem}\label{inv}
Suppose $I$ is a finite set and $S$ and $T$ are involutions on $I$ 
with no fixed points. Then we can write $I$ as a disjoint union
\[
  I=J\amalg K
\]
such that $S(J)=T(J)=K$ (in particular $|K|=|J|=\frac{1}{2}|I|$).
\end{lem}
Now we are ready to prove:
\begin{prop}\label{Z2ses}
Suppose that $G$ is a good finite group, and $N$ is a normal subgroup of $G$ isomorphic to $\Z_2$. Suppose $H=G/N$ possesses a complete mapping. Then $G$ possesses a complete mapping.
\end{prop}
\begin{proof}
Let $N=\{1,x\}$, so that $x$ is a central involution in $G$. Let $\pi:G\twoheadrightarrow H$ be the natural surjection. Clearly 
if $|H|$ is odd, then $N$ is a Sylow-2 subgroup of $G$, contradicting the goodness of $G$. Thus $|H|$ is even. In particular, 
$H$ contains an involution $\bar y$. Then right multiplication by $\bar y$ gives an involution $r_y$ on $H$ with no fixed points.

Now $H$ admits a complete mapping, so choose an indexing set $I$ and bijections $\bar a,\bar b,\bar c:I\rightarrow H$ such 
that $\bar a(i)\bar b(i)=\bar c(i)$ for $i\in I$. Then $S=\bar b^{-1}r_y\bar b$ and $T=\bar c^{-1}r_y\bar c$ are both involutions 
on $I$ with no fixed points. By Lemma \ref{inv}, we can write
\[
  I=J\amalg K
\]
such that $S(J)=T(J)=K$. Now let $y$ be one of the two elements in $\pi^{-1}(\bar y)$. We lift $\bar b$ and $\bar c$ to $G$ as follows. Let
\[
  b,c:J\rightarrow G
\]
be any maps satisfying $\pi b=\bar b$ and $\pi c=\bar c$. Extend $b$ and $c$ to $K$ by defining
\begin{equation}\label{K}
  b(Sj)=b(j)y\hspace{10mm}\text{and}\hspace{10mm}
  c(Tj)=c(j)y\hspace{10mm}\text{for }j\in J.
\end{equation}
By definition of $S$, we have $\bar b(Si)=\bar b(i)\bar y$ for all $i\in I$. Thus
\[
  \pi(b(Sj))=\pi(b(j)y)=\bar b(j)\bar y=\bar b(Sj)
\]
for $j\in J$. With a similar calculation for $c$, we see that $\pi b=\bar b$ and $\pi c=\bar c$ on all of $I$. Finally define $a(i)=c(i)b(i)^{-1}$. 
Then
\[
  \pi(a(i))=\pi(c(i))\pi(b(i))^{-1}=\bar c(i)\bar b(i)^{-1}=\bar a(i),
\]
so that $\pi a=\bar a$. Now define maps $A,B,C:I\times N\rightarrow G$ by
\[\begin{array}{r@{\;=\;}l@{\hspace{10mm}}r@{\;=\;}l@{\hspace{10mm}}r@{\;=\;}l}
  A(j,1)&a(j),&B(j,1)&b(j),&C(j,1)&c(j),\\
  A(j,x)&a(j)x,&B(j,x)&b(j)yx,&C(j,x)&c(j)y,\\
  A(k,1)&a(k),&B(k,1)&b(k)y^{-1}x,&C(k,1)&c(k)y^{-1}x,\text{ and}\\
  A(k,x)&a(k)x,&B(k,x)&b(k),&C(k,x)&c(k)x
\end{array}\]
for $j\in J$ and $k\in K$. Because $a(i)b(i)=c(i)$ for all $i\in I$, and $x$ is central, it is clear that $A(i,t)B(i,t)=C(i,t)$ for all $(i,t)\in I\times N$. 
It remains to show that $A$, $B$ and $C$ are bijective. Since $|I\times N|=|G|$, it is sufficient to prove surjectivity. Since $\pi b=\bar b$ is a 
bijection, $b(I)$ is a transversal for $N$ in $G$, so that $G=b(I)\amalg b(I)x$. Also (\ref{K}) shows that $b(K)=b(J)y$, so
\begin{eqnarray*}
  B(I\times N)
    &=&B(J\times\{1\})\cup B(J\times\{x\})\cup B(K\times\{1\})\cup B(K\times\{x\})\\
    &=&b(J)\cup b(J)yx\cup b(K)y^{-1}x\cup b(K)\\
    &=&b(J)\cup b(K)x\cup b(J)x\cup b(K)\\
    &=&b(I)\cup b(I)x\\
    &=&G,
\end{eqnarray*}
as required. The calculations for $A$ and $C$ are similar.
\end{proof}
Note that Hall and Paige prove Proposition \ref{solu} by proving the HP conjecture for $2$ groups. The above result allows us to easily reproduce the 
conjecture for $2$ groups by induction on the order of the group; indeed if $G$ is a noncyclic $2$ group and $G\not\cong\Z_2\times\Z_2$, one can always 
find a central involution $x\in G$ such that $G/\{1,x\}$ is noncyclic.

Our second version of Proposition \ref{ses}, which deals with a subgroup $N$ in $G$ of index $2$, was proven in \cite{Tony} under the following technical 
assumption: that there exist elements $a,\,b\in G-N$ such that $xa^ix^{-1}\neq b^j$ for all $x\in N$ and odd integers $i,\,j$. We will show that, provided 
$G$ is good, this assumption always holds.

A theorem of Frobenius states that if $n$ divides the order of a finite group $G$, then the number of solutions of $x^n=1$ in $G$ is divisible by $n$ 
(see \cite{Isaacs} for an elementary proof). We will use the following special case. Recall that a 2-element in $G$ is an element whose order is a 
power of $2$.
\begin{lem}\label{p}
Suppose $G$ is a finite group and $P$ is a Sylow-2 subgroup of $G$. Then the number of 2-elements in $G$ is divisible by $|P|$.
\end{lem}
The following result is well known and can be found, for instance, in Theorem 4.2.1 of \cite{Hall}.
\begin{lem}\label{NPH}
Suppose $P$ is a finite 2-group and $H\subset P$ is a proper subgroup. Then the normaliser $N_P(H)$ is strictly larger than $H$.
\end{lem}
\begin{lem}
Suppose $G$ is a good finite group and $N\subseteq G$ is a normal subgroup of index $2$. Consider the cyclic subgroups generated by 2-elements in the set complement $G-N$. These subgroups are not all conjugate in $G$.
\end{lem}
\begin{proof}
Let $X$ be the set of 2-elements in $G-N$. Let
\[
  Y=\{\langle x\rangle\mid x\in X\}
\]
be the set of cyclic groups generated by elements of $X$. For any $x\in X$, we have
\[
  \langle x\rangle-N=\{x^k\mid k\in\Z\text{ odd}\}.
\]
Since $x$ is a 2-element, if $k$ is odd then $x^{kl}=x$ for some $l\in\Z$. Thus $\langle x\rangle$ is generated by any element of $\langle x\rangle-N$. It follows that if $H,H'\in Y$ are distinct, then $(H-N)\cap(H'-N)=\emptyset$. Thus
\[
  X=\coprod_{H'\in Y}(H'-N).
\]
Let $P\subseteq G$ be a Sylow-2 subgroup of $G$. Then $P$ is not contained in $N$, so choose $x\in P-N\subseteq X$, and let $H=\langle x\rangle\in Y$. Suppose by way of contradiction that every $H'\in Y$ is conjugate to $H$. Then the orbit stabiliser theorem gives
\[
  |Y|=\frac{|G|}{|N_G(H)|}.
\]
Also each $H'\in Y$ has the same order as $H$, so
\[
  |X|=\sum_{H'\in Y}|H'-N|
    =\sum_{H'\in Y}\frac{|H'|}{2}
    =\frac{1}{2}|H|\cdot|Y|
    =\frac{|H|\cdot|G|}{2|N_G(H)|}.
\]
Now Lemma \ref{p} shows that $|P|/2$ divides the number of 2-elements in $N$, and the number of 2-elements in $G$. Thus it divides $|X|$, so that
\[
  \frac{|G|}{|P|\cdot[N_P(H):H]}
    =\frac{|H|\cdot|G|}{|N_G(H)|}\cdot\frac{1}{|P|}
      \cdot\frac{|N_G(H)|}{|N_P(H)|}
    =\frac{2|X|}{|P|}\cdot[N_G(H):N_P(H)]\in\Z.
\]
Now $|G|/|P|$ is odd and $[N_P(H):H]$ is a power of $2$, so $[N_P(H):H]=1$. That is, $H=N_P(H)$. By Lemma \ref{NPH}, we must have $P=H$. But $H$ is cyclic, contradicting the goodness of $G$.
\end{proof}
We can now prove the second version of Proposition \ref{ses}:
\begin{prop}\label{sesZ2}
Suppose that $G$ is a good finite group, and $N$ is a normal subgroup of $G$ such that $N$ possesses a complete mapping and $G/N\cong\Z_2$. Then $G$ possesses a complete mapping.
\end{prop}
\begin{proof}
By the previous lemma, we can find 2-elements $a$ and $b$ in $G-N$, such that $\langle a\rangle$ and $\langle b\rangle$ are not conjugate in $G$. For any odd 
integers $i$ and $j$, we have $\langle a^i\rangle=\langle a\rangle$ and $\langle b^j\rangle=\langle b\rangle$. Therefore $\langle a^i\rangle$ is not conjugate to 
$\langle b^j\rangle$, so in particular, $a^i$ and $b^j$ are not conjugate in $G$. The result now follows by Theorem 11 of \cite{Tony}.
\end{proof}
We can now reduce the Hall Paige conjecture to simple groups. The idea of taking a minimal counterexample to the conjecture and considering a minimal normal subgroup is due to Dalla Volta and Gavioli \cite{volta}.
\begin{thm}\label{simple}
Suppose $G$ is a minimal counterexample to the HP conjecture. That is, $G$ is good but has no complete mapping, and any good group smaller than $G$ 
has a complete mapping. Then $G$ is simple.
\end{thm}
\begin{proof}
Suppose otherwise, and let $N$ be a minimal nontrivial normal subgroup of $G$. There are four cases to consider. Suppose first that $N$ and $G/N$ are both good. By the minimality of $G$, they must satisfy the HP conjecture. Thus they both possess complete mappings. Proposition 
\ref{ses} now shows that $G$ possesses a complete mapping.

Next suppose $N$ and $G/N$ are both bad. They are both soluble by Corollary \ref{badcompl}. Thus $G$ is soluble and good, so it possesses a complete mapping by Proposition \ref{solu}.

Now suppose $N$ is good and $G/N$ is bad. If $|N|$ is odd, then $N$ is soluble, and $G$ possesses a complete mapping just as in the last case. Suppose $|N|$ is even. By Lemma \ref{badind2}, there is a characteristic subgroup $\bar H$ of $G/N$ of index $2$. The inverse image $H$ of $\bar H$ is a normal subgroup of $G$ of index $2$ containing $N$. Because the Sylow-2 subgroup of $N$ is noncyclic, the same is true of $H$. Thus $H$ is good, so it possesses a complete mapping by minimality. It follows from Proposition \ref{sesZ2} that $G$ possesses a complete mapping.

Finally suppose $N$ is bad and $G/N$ is good. By Lemma \ref{badind2}, we have a characteristic subgroup $H$ of $N$ of index $2$. Because $H$ is characteristic in $N$, it is normal in $G$. But $N$ is a minimal nontrivial normal subgroup of $G$, so $H$ is trivial. That is, $N\cong\Z_2$. Again $G/N$ possesses a complete mapping by minimality of $G$, so it follows from Proposition \ref{Z2ses} that $G$ possesses a complete mapping.
\end{proof}
\section{Double Coset Results}\label{sec:dbl}
In this section we prove some results similar to Proposition \ref{HPcst}, in which a complete mapping of a subgroup is extended to a complete 
mapping of the group. We will use the following result about transversals, which is an immediate corollary of Theorem 5.1.6 of \cite{HallBook} (see the proof 
of \cite{HallBook} Theorem 5.1.7).
\begin{lem}\label{lrcst}
Suppose $H$ and $K$ are subgroups of $G$ with the same order. Then there exists a left transversal for $H$ which is also a transversal for $K$ (either 
left or right, as desired).
\end{lem}
In fact the proof of Proposition \ref{HPcst} given in \cite{HP} is valid under weaker hypotheses; namely, the elements $x$, $\phi(x)$ and $\psi(x)$ may run 
through three different left transversals of $H$ as $x$ varies, and only one need also be a right transversal. Using this observation and the previous Lemma, 
we may now prove:
\begin{prop}\label{lcst}
Suppose $H\subseteq G$ is a subgroup of $G$ which admits a complete mapping. Suppose we have bijections $\T x$, $\T y$ and $\T z$ from an indexing set $I$ to $G/H$ 
(so that $|I|=[G:H]$), and suppose $\T x(i)\T y(i)\supseteq \T z(i)$ for all $i\in I$. Then $G$ possesses a complete mapping.
\end{prop}
\begin{proof}
By Lemma \ref{lrcst}, we can find a set $\{y_i\mid i\in I\}$ which is both a left and right transversal of $H$. We can label these elements so 
that $y_i\in\T y(i)$. For a given $i\in I$, we have
\[
  \T z(i)\subseteq \T x(i)\T y(i)=\bigcup_{x\in\T x(i)}x\T y(i).
\]
Now each $x\T y(i)$ is a left coset of $H$, and we know distinct left cosets are disjoint, so $\T z(i)=x_i\T y(i)$ for some $x_i\in\T x(i)$. Define 
$z_i=x_iy_i\in x_i\T y(i)=\T z(i)$, so $\{x_i\}$ and $\{z_i\}$ are both left transversals of $H$.

As noted above, the remainder of the proof follows \cite{HP}. Since $H$ possesses a complete mapping, we can choose an indexing set $J$ and 
bijections $a$, $b$ and $c$ from $J$ to $H$ such that $a(j)b(j)=c(j)$ for $j\in J$. Because $\{y_i\}$ is a right transversal of $H$, for any 
$(i,j)\in I\times J$ we can write
\begin{equation}\label{dij}
  y_ia(j)=d(i,j)y_{r(i,j)}
\end{equation}
for some $d(i,j)\in H$ and $r(i,j)\in I$. Then
\[
  z_ic(j)=x_iy_ia(j)b(j)=x_id(i,j)y_{r(i,j)}b(j).
\]
Thus $C(i,j)=A(i,j)B(i,j)$, where the maps $A$, $B$ and $C$ from $I\times J$ to $G$ are defined by
\[
  A(i,j)=x_id(i,j),\hspace{10mm}
  B(i,j)=y_{r(i,j)}b(j),\hspace{10mm}
  C(i,j)=z_ia(j).
\]
It remains to show that these maps are bijective. Since $|I\times J|=|G|$, it suffices to prove injectivity. Suppose $A(i,j)=A(i',j')$, so that 
$x_id(i,j)=x_{i'}d(i',j')$. Now the $x_i$ form a left transversal for $H$ and $d(i,j),\;d(i',j')\in H$, so $i=i'$ and $d(i,j)=d(i',j')$. Thus (\ref{dij}) 
gives
\[
  y_{r(i,j)}a(j)^{-1}=d(i,j)^{-1}y_i=d(i',j')^{-1}y_{i'}=y_{r(i',j')}a(j')^{-1}.
\]
Now the $y_i$ also from a left transversal for $H$, so $a(j)^{-1}=a(j')^{-1}$. Since $a$ is a bijection, this gives $j=j'$, as required.

Now assume $B(i,j)=B(i',j')$, so that $y_{r(i,j)}b(j)=y_{r(i',j')}b(j')$. Since the $y_i$ form a left transversal for $H$, we have $r(i,j)=r(i',j')$ 
and $b(j)=b(j')$. Hence $j=j'$. Now (\ref{dij}) gives
\[
  d(i,j)^{-1}y_i=y_{r(i,j)}a(j)^{-1}=y_{r(i',j')}a(j')^{-1}=d(i',j')^{-1}y_{i'}.
\]
Since the $y_i$ form a right transversal for $H$, we conclude that $i=i'$, as required. The injectivity of $C$ is straightforward.
\end{proof}
The above result is similar to Proposition \ref{ses}; although we are no longer considering a normal subgroup, we require a ``complete mapping" of 
sorts on $G/H$. In fact we have more freedom when $H$ is not normal, since the ``product" of two left cosets $aH$ and $bH$ can be any left coset 
contained in $aHbH$; in general $|aHbH|>|H|$, so there will be more than one choice. The expression $aHbH$ motivates us to consider double cosets; 
recall that a double coset of $H\subseteq G$ is a set of the form $HxH$, for some $x\in G$. We denote the set of double cosets by $\HGH{H}$.
\begin{cor}\label{dcstperm}
Suppose $H\subseteq G$ is a subgroup of $G$ which admits a complete mapping. Suppose $\phi$ and $\psi$ are permutations of $\HGH{H}$ such that 
for each $D\in\HGH{H}$, we have $|D|=|\phi(D)|=|\psi(D)|$ and $D\phi(D)\supseteq\psi(D)$. Then $G$ possesses a complete mapping.
\end{cor}
\begin{proof}
Fix $D\in\HGH{H}$ and pick some $z_D\in\psi(D)\subseteq D\phi(D)$, so that $z_D=x_Dy_D$ for some $x_D\in D$ and $y_D\in\phi(D)$. It is well known that
\[
  D=Hx_DH=\coprod_{h\in X}hx_DH
\]
for any left transversal $X$ of $H\cap x_DHx_D^{-1}$ in $H$. In particular,
\[
  |D|=\frac{|H|^2}{|H\cap x_DHx_D^{-1}|}.
\]
Now $|D|=|\phi(D)|=|\psi(D)|$, so $|H\cap x_DHx_D^{-1}|=|H\cap y_DHy_D^{-1}|=|H\cap z_DHz_D^{-1}|$. By Lemma \ref{lrcst}, we can find 
a subset $X_D\subseteq H$ which is simultaneously a left transversal for $H\cap x_DHx_D^{-1}$ and for $H\cap z_DHz_D^{-1}$ in $H$. Thus
\[
  D=\coprod_{h\in X_D}hx_DH
  \hspace{10mm}\text{and}\hspace{10mm}
  \psi(D)=\coprod_{h\in X_D}hz_DH.
\]
Also let $Y_D$ be a left transversal for $H\cap y_DHy_D^{-1}$ in $H$, so that
\[
  \phi(D)=\coprod_{h\in Y_D}hy_DH.
\]
Now $|X_D|=|Y_D|$, so choose a bijection $\mu_D:X_D\rightarrow Y_D$. Let
\[
  I=\{(D,h)\mid D\in\HGH{H}\text{ and }h\in X_D\}.
\]
Define maps $\T x$, $\T y$ and $\T z:I\rightarrow G/H$ by
\[
  \T x(D,h)=hx_DH,
  \hspace{10mm}
  \T y(D,h)=\mu_D(h)y_DH
  \hspace{5mm}\text{and}\hspace{5mm}
  \T z(D,h)=hz_DH.
\]
Certainly since $\mu_D(h)^{-1}\in H$ for $h\in X_D$, we have
\[
  \T z(D,h)=hz_DH=hx_Dy_DH\subseteq hx_DH\mu_D(h)y_DH=\T x(D,h)\T y(D,h).
\]
Thus the statement will follow from Proposition \ref{lcst}, provided $\T x$, $\T y$ and $\T z$ are bijections. Equivalently, $G$ should be a disjoint 
union of $\T x(D,h)$ for $(D,h)\in I$, and similarly for $\T y$ and $\T z$. Since $\phi$ and $\psi$ are permutations, and $\mu_D$ is a bijection,
\begin{eqnarray*}
  G&=&\coprod_{D\in\HGH{H}}D
    =\coprod_{D\in\HGH{H}\atop h\in X_D}hx_DH
    =\coprod_{(D,h)\in I}\T x(D,h),\\
  G&=&\coprod_{D\in\HGH{H}}\phi(D)
    =\coprod_{D\in\HGH{H}\atop h\in Y_D}hy_DH\\
    &=&\coprod_{D\in\HGH{H}\atop h\in X_D}\mu_D(h)y_DH
    =\coprod_{(D,h)\in I}\T y(D,h)\text{ and}\\
  G&=&\coprod_{D\in\HGH{H}}\psi(D)
    =\coprod_{D\in\HGH{H}\atop h\in X_D}hz_DH
    =\coprod_{(D,h)\in I}\T z(D,h).
\end{eqnarray*}
\end{proof}
We will use the special case in which $\phi$ and $\psi$ are the identity:
\begin{cor}\label{dcst}
Suppose $H\subseteq G$ is a subgroup of $G$ which admits a complete mapping. Suppose $D^2\supseteq D$ for all $D\in\HGH{H}$. Then $G$ possesses a 
complete mapping.
\end{cor}
\section{The HP Conjecture for Groups of Lie Type}\label{sec:lie}
In this section we suppose that $G$ is a finite simple group of Lie type, excluding the Tits group. We begin by stating a number of properties 
of such groups, which can be found in \cite{Carter}. First we recall below the list of families of such groups.

\begin{center}
\begin{tabular}{l*{2}{@{\quad}l}}
	Group&Parameter values&Rank $l$\\ \hline
	$A_k(q)$&$k\geq1$ and $q$ a prime power&$k$\\
	$B_k(q)$&$k\geq2$ and $q$ a prime power&$k$\\
	$C_k(q)$&$k\geq3$ and $q$ a prime power&$k$\\
	$D_k(q)$&$k\geq4$ and $q$ a prime power&$k$\\
	$E_6(q)$&$q$ a prime power&$6$\\
	$E_7(q)$&$q$ a prime power&$7$\\
	$E_8(q)$&$q$ a prime power&$8$\\
	$F_4(q)$&$q$ a prime power&$4$\\
	$G_2(q)$&$q$ a prime power&$2$\\
	${}^2A_k(q)$&$k\geq2$ and $q$ a prime power squared&$\lceil k/2\rceil$\\
	${}^2D_k(q)$&$k\geq4$ and $q$ a prime power squared&$k-1$\\
	${}^3D_4(q)$&$q$ a prime power cubed&$2$\\
	${}^2E_6(q)$&$q$ a prime power squared&$4$\\
	${}^2B_2(q)$&$q=2^{2k+1}$&$1$\\
	${}^2F_4(q)$&$q=2^{2k+1}$&$2$\\
	${}^2G_2(q)$&$q=3^{2k+1}$&$1$\\ \hline
\end{tabular}

Table 1: Families of Groups of Lie Type
\end{center}

% ------------------------------- Untwisted groups
Consider first the \emph{untwisted groups}, namely those with no superscript on the left. Let $\Phi=\Phi^+\cup\Phi^-$ denote the corresponding 
root system, and let $K$ denote the field of $q$ elements. The group $G$ is generated by elements $x_r(t)$ for $r\in\Phi$ and $t\in K$, which 
satisfy $x_r(t)x_r(u)=x_r(t+u)$. If $r,\,s\in\Phi$ are linearly independent, then $x_s(u)$ and $x_r(t)$ satisfy Chevalley's commutator formula 
(\cite{Carter} Theorem 5.2.2):
\begin{equation}\label{Chev}
	x_r(t)x_s(u)=x_s(u)\left[\prod_{i,j>0,\atop ir+js\in\Phi}x_{ir+js}(C_{ijrs}t^iu^j)\right]x_r(t),
\end{equation}
where the product is taken in order of increasing $i+j$, and the integer constants $C_{ijrs}$ are determined by
\[\begin{array}{r@{\,=\,}l@{\hspace{10mm}}r@{\,=\,}l}
	C_{i1rs}&M_{rsi};&C_{1jrs}&(-1)^jM_{srj};\\
	C_{32rs}&\frac{1}{3}M_{r+s,r,2};&C_{23rs}&-\frac{2}{3}M_{s+r,s,2};\\
	M_{rsi}&\frac{1}{i!}\prod_{j=0}^{i-1}N_{r,jr+s};&N_{r,s}&\pm\max\{p+1\mid s-pr\in\Phi\}.
\end{array}\]
The signs of $N_{r,s}$ are chosen to satisfy certain conditions that do not concern us here (\cite{Carter} Sections 4.2). We 
will only use (\ref{Chev}) for $G_2$, in which case we use the values for $N_{r,s}$ at the end of Section 12.4 of \cite{Carter}.

For $r\in\Phi$ we have a homomorphism $\rho_r:SL_2(K)\rightarrow G$ satisfying
\[
	\rho_r\begin{pmatrix}1&t\\0&1\end{pmatrix}=x_r(t)
	\hspace{10mm}\text{and}\hspace{10mm}
	\rho_r\begin{pmatrix}1&0\\t&1\end{pmatrix}=x_{-r}(t).
\]
Let
\[
	h_r(t)=\rho_r\begin{pmatrix}
		t&0\\
		0&t^{-1}
	\end{pmatrix}
	\hspace{10mm}\text{and}\hspace{10mm}
	n_r=\rho_r\begin{pmatrix}
		0&1\\
		-1&0
	\end{pmatrix},
\]
for $t\in K^*$. Clearly $n_r^2=h_r(-1)$. Let $H$, $U$, $V$, $B$ and $N$ be the subgroups of $G$ generated by:
\begin{eqnarray*}
	H&=&\langle h_r(t)\mid r\in\Phi,\,t\in K^*\rangle\subseteq G,\\
	U&=&\langle x_r(t)\mid r\in\Phi^+,\,t\in K\rangle\subseteq G,\\
	V&=&\langle x_r(t)\mid r\in\Phi^-,\,t\in K\rangle\subseteq G,\\
	B&=&\langle H,U\rangle\subseteq G,\\
	N&=&\langle h_r(t),\,n_r\mid r\in\Phi,\,t\in K^*\rangle\subseteq G.
\end{eqnarray*}
We will need an explicit description of the subgroup $H$. Let $\Lambda$ denote the lattice spanned by $\Phi$; this is a free abelian group 
whose rank is the rank $l$ in Table 1.  Let $\Lambda^*$ denote the dual lattice to $\Lambda$. The Cartan matrix of $\Phi$ gives a bilinear 
form on $\Lambda$ (not necessarily symmetric) which allows us to identify $\Lambda$ with a subgroup of $\Lambda^*$. Then $H$ can 
be identified with the image of $\Hom(\Lambda^*,K^*)$ in $\Hom(\Lambda,K^*)$ (see \cite{Carter} Section 7.1). Explicitly, $h_r(t)$ 
corresponds to the function
\[
	h_r(t)(v)=t^{\frac{2(r,v)}{(r,r)}}\text{ for }v\in\Lambda,
\]
where $(\cdot,\cdot)$ is a symmetrised version of the above bilinear form. In particular, if $r,\,s,\,ir+js\in\Phi$ and $t\in K^*$, then
\begin{equation}\label{hr+s}
	h_{ir+js}(t)^{(ir+js,ir+js)}=h_r(t)^{i(r,r)}h_s(t)^{j(s,s)}.
\end{equation}
We have a left exact sequence
\begin{equation}\label{les}
	\Hom(\Lambda^*/\Lambda,K^*)\hookrightarrow\Hom(\Lambda^*,K^*)\rightarrow\Hom(\Lambda,K^*),
\end{equation}
so $H\cong\hat{H}/H_1$, where $\hat{H}=\Hom(\Lambda^*,K^*)\cong(K^*)^{\oplus l}$ and $H_1=\Hom(\Lambda^*/\Lambda,K^*)$. The group 
$\Lambda^*/\Lambda$ is given by the following table (see Section 8.6 of \cite{Carter}):
\[\begin{array}{r|*{10}{@{\quad}l}}
	G&A_k&B_k&C_k&D_{2k+1}&D_{2k}&G_2&F_4&E_6&E_7&E_8\\
	\Lambda^*/\Lambda&\Z_{k+1}&\Z_2&\Z_2&\Z_4&\Z_2\oplus\Z_2&0&0&\Z_3&\Z_2&0
\end{array}\]
In particular, $\Lambda^*/\Lambda$ is generated by at most two elements in the case of $D_{2k+1}$, and is cyclic otherwise. The same statement 
is true of $H_1$, since $K^*$ is cyclic.

% ----------------------------Twisted groups
Now suppose $G$ is a \emph{twisted group}, that is, one with a superscript on the left. Let $\bar G$ be the corresponding untwisted group 
obtained by removing the superscript, with subgroups $\bar H$, $\bU$, $\bV$, $\bar B$ and $\bar N$ as constructed above. There is an 
automorphism $\sigma$ of $\bar G$ (described in \cite{Carter} Section 13.4) fixing all of these subgroups. Let $U=\bU^\sigma$ and 
$V=\bV^\sigma$ be the groups of $\sigma$-invariant elements in $\bU$ and $\bV$. Then $G$ is defined to be the subgroup of $\bar G$ 
generated by $U$ and $V$. Let$H=\bar H\cap G$, and similarly for $B$ and $N$. Again we will need to describe $H$ more explicitly. First 
suppose that $\bar G$ is simply laced; that is, $G$ is not ${}^2B_2(q)$, ${}^2F_4(q)$ or ${}^2G_2(q)$. Recall $\bar H$ is the image of the 
second map in (\ref{les}). The action of $\sigma$ on $\bar H$ can be extended to $\Hom(\Lambda,K^*)$ by
\[
	\sigma(\chi)(r)=\chi(\tau r)^\theta,
\]
where $\tau$ is an isometry of $\Lambda$ fixing the set of simple roots $\Pi\subseteq\Phi$, and $\theta$ is an automorphism of $K$ 
(see \cite{Carter} Lemma 13.7.1). Similarly $\sigma$ acts on $\Hom(\Lambda^*,K^*)$ and $\Hom(\Lambda^*/\Lambda,K^*)$, and we 
obtain a left exact sequence
\[
	\Hom(\Lambda^*/\Lambda,K^*)^\sigma\hookrightarrow\Hom(\Lambda^*,K^*)^\sigma\rightarrow\Hom(\Lambda,K^*)^\sigma.
\]
In fact $H$ is the image of the second map (\cite{Carter} Theorem 13.7.2), so we can again write $H$ as a quotient
\[
	H\cong\hat{H}/H_1,
\]
now with $\hat{H}=\Hom(\Lambda^*,K^*)^\sigma$ and $H_1=\Hom(\Lambda^*/\Lambda,K^*)^\sigma$. Since $\tau$ permutes the set $\Pi^*$, which 
freely generates $\Lambda^*$  as an abelian group, we have
\[
	\hat{H}\cong\bigoplus_{O\in\tau\backslash\Pi^*}(K^*)^{\theta^{|O|}},
\]
where $\tau\backslash\Pi^*$ is the set of orbits of $\tau$ in $\Pi^*$, and $(K^*)^{\theta^{|O|}}$ is the subgroup of $K^*$ fixed by 
$\theta^{|O|}$. The number of orbits $l=|\tau\backslash\Pi^*|$ is exactly the rank listed in Table 1. If $\bar G$ is not of type $D_{2k}$, 
then $\Hom(\Lambda^*/\Lambda,K^*)$ is cyclic as noted above, so the same is true of the subgroup $H_1$. In fact $H_1$ is trivial 
when $G$ has type ${}^3D_4(q)$, and is trivial or $\Z_2$ when $G$ has type ${}^2D_{2k}(q)$ (see the note after \cite{Carter} Lemma 14.1.2).

Finally suppose $G$ is of type ${}^2B_2(q)$, ${}^2F_4(q)$ or ${}^2G_2(q)$. Let $p$ be the characteristic of $K$. For ${}^2B_2(q)$ we have 
$\Lambda^*/\Lambda\cong\Z_2$ and $p=2$, and otherwise $\Lambda^*/\Lambda=0$, so
\[
	\Hom(\Lambda^*/\Lambda,K^*)=\Ext^1(\Lambda^*/\Lambda,K^*)=0
\]
in either case. Hence
\[
	\bar H=\Hom(\Lambda^*,K^*)=\Hom(\Lambda,K^*).
\]
We again have a permutation $\tau$ of $\Pi$; this no longer induces an isometry as the roots of $\Pi$ have different lengths. Nevertheless 
we have the following explicit description of $H$ (\cite{Carter} Theorem 13.7.4):
\begin{equation}\label{Htwisted}
	H=\{\chi\in\bar H\mid\chi(r)=\chi(\tau r)^{(r,r)\theta}\text{ for }r\in\Pi\},
\end{equation}
where $q=p\theta^2$ and $(\cdot,\cdot)$ is normalized to give short roots length $1$. Also $\tau$ acts on $\Pi$ as an involution switching 
long and short roots, and $l=|\tau\backslash\Pi|$ is the rank listed in Table 1. It follows that
\[
	H\cong(K^*)^{\oplus l}.
\]
In this case we set $\hat H=H$ and $H_1=0$.
%-----------------------Common properties

We now state some results applicable to every group $G$ in Table 1. In each case we have written $H$ as a quotient $\hat H/H_1$, where 
$\hat H$ is a product of $l$ cyclic groups, each of order $p^s-1$ for some $s$, where $p$ is the characteristic of $K$. The subgroups 
$U$, $V$, $B$, $H$ and $N$ in $G$, constructed above, satisfy $H=B\cap N$ and $B=H\ltimes U$. It is shown in \cite{Carter} Sections 8.6 
and 14.1 that
\begin{equation}\label{|U|}
	|U|=p^N\text{ for some integer }N\text{, and }p\nmid[G:U].
\end{equation}

The subgroups $B$ and $N$ form a $(B,N)$-pair in $G$. The following results about $(B,N)$-pairs are proved in \cite{bourbaki}. Firstly $H$ 
is normal in $N$, and the quotient $W=N/H$ is a Coxeter group, generated by a set $S$ of involutions. In fact $l=|S|$ is the rank listed in 
Table 1. Moreover if $G$ is untwisted, we may take $S$ to be the image of $\{n_r\mid r\in\Pi\}$, and $W$ is the Coxeter group with the same 
Dynkin diagram as $G$. If $G$ is twisted, $W$ is either dihedral or of type $A_1$, $B_l$ or $F_4$. The double cosets of $B$ in $G$ are 
indexed by $W$, so that
\[
  G=\coprod_{w\in W}BwB.
\]
Here, by abuse of notation, we use $BwB$ to denote $B\bar wB$ for any $\bar w\in N$ which maps to $w\in W=N/H$. In fact we can say more in 
the case of groups of Lie type. For each $w\in W$, choose $n_w\in N$ mapping to $w$. Then every element of $G$ has a unique expression of 
the form
\begin{equation}\label{normalform}
	u'n_whu,
\end{equation}
where $w\in W$, $u\in U$, $h\in H$ and $u'$ is in a subgroup $U_w^-$ of $U$ (\cite{Carter} Corollary 8.4.4 and Proposition 13.5.3). 
We will not use any properties of $U_w^-$, except that $U_1^-$ is trivial and $U_w^-=U$ when $w$ is the longest element of $W$.

Let $\ell:W\rightarrow\Z_{\geq0}$ denote the usual length function. The product of two double cosets of $B$ in $G$ is determined by the formula
\begin{equation}\label{BsBBwB}
	(BsB)(BwB)=\begin{cases}
			BswB&\text{if }\ell(sw)>\ell(w),\\
			BswB\cup BwB&\text{if }\ell(sw)<\ell(w),
		\end{cases}\hspace{5mm}\text{for }s\in S\text{ and }w\in W.
\end{equation}
A subgroup $W'$ of $W$ is called a \emph{parabolic subgroup} if it is generated by $W'\cap S$. In this case each double coset in 
$\HGH[W]{W'}$ contains a unique element of minimal length; these elements are the \emph{minimal double coset representatives} for 
$W'$ in $W$. Also the subset
\[
  P=\coprod_{w\in W'}BwB
\]
is a subgroup of $G$, also called a parabolic subgroup. There is a natural correspondence between $\HGH[W]{W'}$ and $\HGH{P}$. Explicitly, 
every double coset in $\HGH{P}$ can be written as
\begin{equation}\label{PaP}
	PaP=\coprod_{w\in W'vW'}BwB
\end{equation}
for some $W'vW'\in\HGH[W]{W'}$.

Now suppose $r\in S$, and let $W'\subseteq W$ be the parabolic subgroup generated by $S-\{r\}$. Let $P$ be the corresponding parabolic subgroup 
of $G$ as above. Our aim is to apply Corollary \ref{dcst} to $P\subseteq G$. We first prove:
\begin{lem}\label{Pgood}
We can choose $r$ so that $P$ is good, except when $G$ has type ${}^2G_2$, ${}^2A_2$ or $A_1$, and $q$ is odd and $|H|$ is even. Moreover 
any $r$ will suffice, except for types $B_2$ and ${}^2A_3$.
\end{lem}
\begin{proof}
First suppose $q$ is even. Then $|U|$ is a power of $2$ and $[G:U]$ is odd, by (\ref{|U|}). Hence $U$ is a Sylow-2 subgroup of $G$. 
It is noncyclic since $G$ is good. Since $U\subseteq B\subseteq P$, this shows that $P$ is good.

Now suppose $q$ is odd. Note that this excludes types ${}^2B_2$ and ${}^2F_4$. Suppose first that $l=1$, so that $G$ has type 
${}^2G_2$, ${}^2A_2$ or $A_1$. In these cases, we are not required to prove the statement when $|H|$ is even, so suppose $|H|$ is odd (in 
fact this can only occur for type $A_1$). Since $l=1$ for these groups, $W'$ is trivial and $P=B$. But then $|P|=|H|\cdot|U|$ is odd by 
(\ref{|U|}), so $P$ is good.

Now suppose $l\geq2$. It suffices to prove that $P$ contains two nontrivial commuting involutions, as this would prevent the Sylow-2 
subgroup of $P$ from being cyclic. Since $H\subseteq B\subseteq P$ and $H$ is abelian, it suffices to prove that $H$ contains two nontrivial 
involutions. Let $\#(H)$ denote the number of involutions of $H$ (including the identity).

Recall that $H$ is a quotient $H=\hat H/H_1$. It is easy to see that $\#(\hat H)\leq\#(H)\#(H_1)$. Moreover since $p$ is odd, $\hat H$ is 
a direct product of $l$ cyclic groups of even order, so $\#(\hat H)=2^l$. If $G$ is of type $D_k$, then $l=k\geq4$ and $H_1$ is generated 
by at most two elements, so
\[
	16\leq\#(\hat H)\leq\#(H)\#(H_1)\leq 4\times\#(H).
\]
Hence $\#(H)\geq4$, as required. In any other case, $H_1$ is cyclic, so $\#(H_1)\leq2$, giving $\#(H)\geq2^{l-1}$. If $l\geq3$, we again 
obtain $\#(H)\geq4$. We are left with the rank $2$ cases, namely $A_2$, $B_2$, $G_2$, ${}^2A_3$, ${}^2A_4$ and ${}^3D_4$. For types $A_2$, 
$G_2$ and ${}^2A_4$, the group $\Lambda^*/\Lambda$ has odd order, so the same is true of $H_1$. For type ${}^3D_4$, the group $H_1$ is 
trivial as noted above. In these cases, $\#(H_1)=1$, so $\#(H)\geq\#(\hat H)=4$, as required.

The remaining cases, $B_2$ and ${}^2A_3$, are dealt with most easily by realising the group $G$ explicitly. First consider the $B_2$ 
case. Let
\[
	X=\begin{pmatrix}0&0&1&0\\0&0&0&1\\-1&0&0&0\\0&-1&0&0\end{pmatrix},
\]
and consider the group
\[
	M=\{A\in SL_4(K)\mid A^tXA=X\},
\]
where $A^t$ denotes the transpose of $A$. Let $Z=\{\pm1\}$ denote the subgroup of scalar matrices in $M$. There is an isomorphism 
$G\cong M/Z$ mapping $H$ to the image of the diagonal matrices in $M$ (see \cite{Carter} Theorem 11.3.2 (iii)). Thus $H$ contains the 
image of
\[
	\begin{pmatrix}1&0&0&0\\0&-1&0&0\\0&0&1&0\\0&0&0&-1\end{pmatrix}.
\]
Unfortunately, $H$ may not be good in this case. However, $P$ must also contain the double coset of $B$ corresponding to one of the elements 
of $S$. By choosing $r$ appropriately, we may suppose $P$ contains the image of
\[
	\begin{pmatrix}0&1&0&0\\-1&0&0&0\\0&0&0&1\\0&0&-1&0\end{pmatrix}.
\]
The images of these elements in $M/Z$ are distinct commuting involutions, and we are done.

Now consider the ${}^2A_3$ case. Recall that $q=|K|$ is a prime power squared, so $K$ is a degree $2$ extension of a subfield $L$. Let 
$\bar\;$ be the unique nontrivial automorphism of $K$ over $L$. Let
\[
	X=\begin{pmatrix}0&0&0&1\\0&0&-1&0\\0&1&0&0\\-1&0&0&0\end{pmatrix},
\]
and consider the group
\[
	M=\{A\in SL_4(K)\mid A^\dagger XA=X\},
\]
where $A^\dagger$ denotes the conjugate transpose of $A$ with respect to $\bar\;$. Again we have $G\cong M/Z$, where $Z$ is the subgroup 
of scalar matrices in $M$, and $H$ maps to the image of diagonal matrices (see \cite{Carter} Theorem 14.5.1). Now $H$ contains the image of
\[
	\begin{pmatrix}1&0&0&0\\0&-1&0&0\\0&0&-1&0\\0&0&0&1\end{pmatrix}.
\]
Again by choosing $r$ appropriately, we may ensure that $P$ contains the image of
\[
	\begin{pmatrix}0&1&0&0\\-1&0&0&0\\0&0&0&1\\0&0&-1&0\end{pmatrix}.
\]
As above, the images are distinct commuting involutions, and we are done.
\end{proof}
For the next proof, we require two easy consequences of (\ref{BsBBwB}). We say that the expression $u_1u_2\ldots u_k\in W$ is \emph{reduced} 
if
\[
	\ell(u_1u_2\ldots u_k)=\ell(u_1)+\ell(u_2)+\ldots+\ell(u_k).
\]
It follows inductively from (\ref{BsBBwB}) that if $s_1s_2\ldots s_k$ is a reduced expression, with $s_i\in S$, then
\[
	Bs_1s_2\ldots s_kB=(Bs_1B)(Bs_2B)\ldots(Bs_kB).
\]
Hence
\begin{equation}\label{BwuB}
	(BwB)(BuB)=BwuB\text{ whenever }wu\text{ is reduced}.
\end{equation}
It follows that
\begin{equation}\label{BwBBuB}
	(BwB)(BuB)=(BuB)(BwB)\text{ if }wu=uw\text{ is reduced}.
\end{equation}
\begin{lem}\label{P^2}
We can choose $r$ so that every double coset $D\in\HGH{P}$ satisfies $D^2\supseteq D$. In the cases of $B_2$ and ${}^2A_3$, either choice 
of $r$ will suffice.
\end{lem}
\begin{proof}
Because $G$ is a disjoint union of double cosets of $P$, we need only show that $D^2$ intersects $D$ for each $D\in\HGH{P}$. By (\ref{PaP}), 
it suffices to show that every double coset in $\HGH[W]{W'}$ contains an element $w$ satisfying
\[
	(BwB)^2\supseteq BwB.
\]
This will follow if $w$ has a reduced expression of the form
\begin{equation}\label{yuck}
	w=u^{-1}s_1s_2\ldots s_pu,
\end{equation}
where $u\in W$ and the $s_i$ are commuting elements of $S$ (we allow $p$ to be $0$ and $u$ to be the identity). Indeed,
\begin{eqnarray*}
	(BwB)(BwB)
		&=&(Bu^{-1}B)(Bs_1B)\ldots(Bs_pB)(BuB)\\
		&&\;\;\times(Bu^{-1}B)(Bs_1B)\ldots(Bs_pB)(BuB)\hspace{10mm}\text{by (\ref{BwuB})}\\
		&\supseteq&(Bu^{-1}B)(Bs_1B)\ldots(Bs_pB)(Bs_1B)\ldots(Bs_pB)(BuB)\\
		&=&(Bu^{-1}B)(Bs_1B)^2\ldots(Bs_pB)^2(BuB)
			\hspace{10mm}\text{by (\ref{BwBBuB})}\\
		&\supseteq&(Bu^{-1}B)(Bs_1B)\ldots(Bs_pB)(BuB)
			\hspace{10mm}\text{by (\ref{BsBBwB})}\\
		&=&BwB
			\hspace{10mm}\text{by (\ref{BwuB}).}
\end{eqnarray*}
We will consider each possibility for $W$ and, for a particular choice of $r$, find a set of double coset representatives for $W'$ in $W$, 
each with a reduced expression of the form (\ref{yuck}); in fact it will be the set of minimal coset representatives in each case.

\textbf{Case 1 - Dihedral group:}
Suppose that $l=2$; that is, $W$ is the dihedral group. Then $S=\{r,s\}$ and $rs$ has order $n$, where $|W|=2n$. Now $W'=\{1,s\}$, so the 
minimal double coset representatives are
\[
	1,\;r,\;rsr,\;rsrsr,\;\ldots,\;rsrs\ldots srsr,
\]
where the length of the last word is $n$ or $n-1$. All these words are of the form (\ref{yuck}), as follows. For the identity we take $p=0$ and 
$u=1$. For the rest we take $p=1$; either $s_1=r$ and $u=(sr)^i$, or $s_1=s$ and $u=r(sr)^i$.

Note that our argument did not depend on the choice of $r\in S$; indeed there is an automorphism of $W$ switching the elements of 
$S$. In particular this applies in the cases $B_2$ and ${}^2A_3$, both of which have rank $2$.

\textbf{Case 2 - Type $A_l$:}
In this case $W$ is the symmetric group $S_{l+1}$. Choose $r$ to be the rightmost node, so $W'$ is the natural copy of $S_l$ 
in $S_{l+1}$. It is easy to see that there are just two double cosets of $W'$ in $W$; the minimal coset representatives are $1$ and 
$r$, both of which are of the form (\ref{yuck}).

\textbf{Case 3 - Type $B_l=C_l$:}
Now $W$ is the wreath product
\[
	W=S_l\wr\Z_2=\{(\sigma,\epsilon_1,\epsilon_2,\ldots,\epsilon_l)\mid\sigma\in S_l,\;\epsilon_i=\pm1\}.
\]
Write $S=\{\tau,s_1,s_2,\ldots,s_{l-1}\}$, where the $s_i$ generate $S_l$, and
\[
	\tau=(1,-1,1,1,\ldots,1).
\]
Let $r=s_{l-1}$. Consider the double coset of $W'$ in $W$ containing $u=(\sigma,\epsilon_1,\epsilon_2,\ldots,\epsilon_l)$. As in Case 2, 
by multiplying $u$ on the left and right by elements of $S_{l-1}$, we may suppose that $\sigma=1$ or $\sigma=r$. Also for any 
$\nu_1,\nu_2,\ldots,\nu_{l-1}\in\{\pm1\}$, the element $(1,\nu_1,\nu_2,\ldots,\nu_{l-1},1)$ is in $W'$. Hence if $\sigma=1$, we may replace 
$u$ by
\[
	u(1,\epsilon_1,\epsilon_2,\ldots,\epsilon_{l-1},1)
		=(1,1,\ldots,1,\epsilon_l).
\]
If $\sigma=r$, we may replace $u$ by
\[
	(1,1,\ldots,\epsilon_l,1)u(1,\epsilon_1,\epsilon_2,\ldots,\epsilon_{l-1},1)
		=(r,1,\ldots,1,1).
\]
Therefore the elements $(1,1,\ldots,1,\pm1)$ and $(r,1,\ldots,1)$ form a set of double coset representatives for $W'$ in $W$. Written in terms 
of the generators, these elements are
\[
	1,\;r,\;s_{l-1}s_{l-2}\ldots s_1\tau s_1s_2\ldots s_{l-1}.
\]
Again each is of the form (\ref{yuck}).

\textbf{Case 3 - Type $D_l$:}
We may realise $W$ as the subgroup
\[
	W=\{(\sigma,\epsilon_1,\epsilon_2,\ldots,\epsilon_l)\in S_l\wr\Z_2\mid\epsilon_1\epsilon_2\ldots\epsilon_l=1\}.
\]
The generating involutions are $\{\rho,s_1,\ldots,s_{l-1}\}$, where the $s_i$ are as above, but $\rho$ is now $(s_1,-1,-1,1,\ldots,1)$. Again 
let $r=s_{l-1}$. Arguing as in the previous case, the double cosets of $W'$ in $W$ are represented by the elements
\begin{eqnarray*}
	(1,1,1,\ldots,1,1)&=&1,\\
	(1,-1,1,\ldots,1,-1)&=&s_{l-1}s_{l-2}\ldots s_2s_1\rho s_2s_3\ldots s_{l-1},\\
	(r,1,1,\ldots,1,1)&=&r.
\end{eqnarray*}
As above, these expressions are of the form (\ref{yuck}).

For a specfic Coxeter group, a computer algebra package such as MAGMA \cite{MAGMA} can be used to find the minimal double coset 
representatives for $W'$ in $W$, and to determine when specified words are reduced. We do so in the remaining cases without further 
comment. Also for brevity, we simply denote elements of $S$ by integers, and we use $\varepsilon$ to denote the identity.

% M:=SymmetricMatrix([ 1, 3,1, 2,3,1, 2,2,3,1, 2,2,3,2,1, 2,2,2,2,3,1]);
% G:=CoxeterGroup(M);
% H:=StandardParabolicSubgroup(G,{1..5});
% T:=Transversal(G,H);
% D:=T meet {w^(-1) : w in T};
% W,f:=CoxeterGroup(GrpFPCox,G);
% {f(w) : w in D};
\textbf{Case 4 - $W=E_6$:}
Label $S$ as shown below.
% Clear[f, r, ls]
% r = .1;
% f[z_] = {Re[z], Im[z]};
% ls = {-2, -1, 0, I, 0, 1, 2};
% labels = {{-2, 1, I}, {-1, 2, I}, {0, 3, I}, {I, 4, 1}, {1, 5, I}, {2, 6, I}};
% Show[Graphics[{Table[
%           Line[{{1 - r, r}, {r, 1 - r}}.f /@ ls[[{i, i + 1}]]], {i, 
%             Length[ls] - 1}], Circle[f[#], r] & /@ ls, 
%         Text[#[[2]], f[{1, 0, -r}.#], 1.1 f[#[[3]]]] & /@ labels
%         }], PlotRange -> All, AspectRatio -> Automatic, 
%     TextStyle -> {FontSize -> 16, FontFamily -> Times}];
\[\includegraphics[height=15mm]{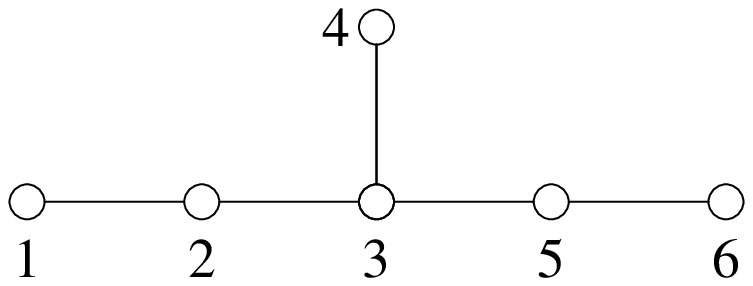}\]
Let $r=6$. The minimal double coset representatives are
\[
	\varepsilon,\;6,\;65324356.
\]
The last representative can be written $u^{-1}24u$, with $u=356$. Since $2$ and $4$ commute, each representative is of the form (\ref{yuck}), 
and we are done.

\textbf{Case 5 - $W=E_7$:}
Label $S$ as shown below.
% ls = {-2, -1, 0, I, 0, 1, 2, 3};
% labels = {{-2, 1, I}, {-1, 2, I}, {0, 3, I}, {I, 4, 1}, {1, 5, I}, {2, 6, I}, {3, 7, I}};
\[\includegraphics[height=15mm]{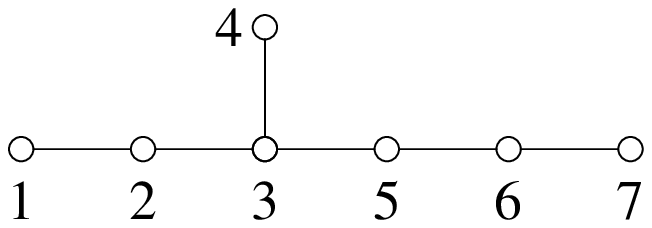}\]
Let $r=7$. The minimal double coset representatives are
\[
	\varepsilon,\;7,\;7653243567,\;765321432534653217653243567.
\]
The last two words can be written $u_1^{-1}24u_1$ and $u_2^{-1}457u_2$ respectively, where
\begin{eqnarray*}
	u_1&=&3567,\\
	u_2&=&635234123567.
\end{eqnarray*}
Since $2$, $4$, $5$ and $7$ all commute, again each word is of the form (\ref{yuck}).

\textbf{Case 6 - $W=E_8$:}
Label $S$ as shown below.
% ls = {-2, -1, 0, I, 0, 1, 2, 3, 4};
% labels = {{-2, 1, I}, {-1, 2, I}, {0, 3, I}, {I, 4, 1}, {1, 5, I}, {2, 6, I}, {3, 7, I}, {4, 8, I}};
\[\includegraphics[height=15mm]{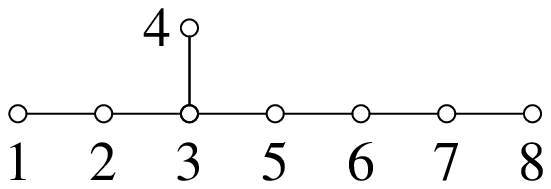}\]
Let $r=8$. The minimal double coset representatives are
\begin{eqnarray*}
	&\varepsilon,\;8,\;876532435678,\;87653214325346532176532435678,&\\
		&876532143253465321765324356787653214325346532176532435678.&
\end{eqnarray*}
The last three words can be written $u_1^{-1}24u_1$, $u_2^{-1}457u_2$ and $u_3^{-1}8u_3$ respectively, where
\begin{eqnarray*}
	u_1&=&35678,\\
	u_2&=&6352341235678,\\
	u_3&=&7653423567123564352341235678.
\end{eqnarray*}
Again since $2$, $4$, $5$ and $7$ all commute, each word is of the form (\ref{yuck}).

\textbf{Case 7 - $W=F_4$:}
Label $S$ as shown below.
% Clear[r]
% r = .1;
% Show[Graphics[{
%         Line[{{1 + r, 0}, {2 - r, 0}}],
%         Line[{{2, r/2}, {3, r/2}}],
%         Line[{{2, -r/2}, {3, -r/2}}],
%         Line[{{3 + r, 0}, {4 - r, 0}}],
%         {RGBColor[1, 1, 1], Disk[{#, 0}, r] & /@ Range[2, 3], 
%           Point[{1, -4 r}]},
%         Circle[{#, 0}, r] & /@ Range[4],
%         Text[#, {#, -r}, {0, 1}] & /@ Range[4]
%         }], PlotRange -> All, AspectRatio -> Automatic, 
%     TextStyle -> {FontSize -> 16, FontFamily -> Times}];
\[\includegraphics[width=40mm]{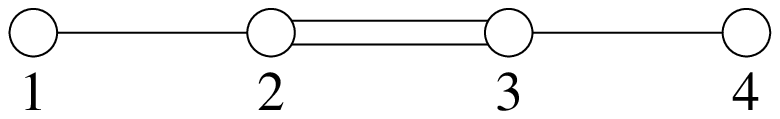}\]
Let $r=4$. The minimal double coset representatives are
\[
	\varepsilon,\;4,\;43234,\;43213234,\;432132343213234.
\]
The last three words can be written $u_1^{-1}2u_1$, $u_2^{-1}13u_2$ and $u_3^{-1}4u_3$ respectively, where
\begin{eqnarray*}
	u_1&=&34,\\
	u_2&=&234,\\
	u_3&=&3213234.
\end{eqnarray*}
Since $1$ and $3$ commute, each word is of the form (\ref{yuck}).
\end{proof}
The previous two lemmas allow us to apply Corollary \ref{dcst}, except when $G$ has type $A_1$, ${}^2A_2$ or ${}^2G_2$, 
and $q$ is odd and $|H|$ is even. From this point, suppose $G$ is such a group. Then the Weyl group of $G$ has rank 1; that is, $|W|=2$. 
Unfortunately this implies that there is only one proper parabolic subgroup, namely $B$, and it is bad. We are forced to do some explicit 
calculations in these cases.

\begin{lem}
There is an involution $n\in N-H$, and $nUn=V$. Every element of $G$ is uniquely expressible either as $hu$ or $u'nhu$, where $h\in H$ 
and $u,u'\in U$.
\end{lem}
\begin{proof}
The second statement is true for any $n\in N-H$ by (\ref{normalform}). For type $A_1$, the group $G$ is $PSL_2(K)$ (\cite{Carter} 
Theorem 11.3.2), and we may take
\[
	n=\begin{pmatrix}0&1\\ -1&0\end{pmatrix}.
\]
In the ${}^2A_2$ case, $K$ is a degree 2 extension of a subfield $L$. Let $\bar\;$ denote the nontrivial automorphism of $K$ over $L$. Let 
\[
	X=\begin{pmatrix}0&0&1\\ 0&-1&0\\ 1&0&0\end{pmatrix}
\]
and $M=\{A\in SL_3(K)\mid A^\dagger XA=X\}$, where $\dagger$ denotes conjugate transpose with respect to $\bar\;$. Then $G$ is $M$ modulo 
scalar matrices (\cite{Carter} Theorem 14.5.1), and we may take $n$ to be the image of $X$.

Finally suppose $G={}^2G_2(q)$, where $q=3^{2k+1}=3\theta^2$. Recall that $G$ is a subgroup of the group $\bar G$ of 
type $G_2(q)$. Let $r$ and $s$ denote, respectively, the short and long simple roots of the root system $\Phi$ of $\bar G$. For brevity we 
will denote
\[
	x_{ij}(t)=x_{ir+js}(t),\hspace{5mm}
	y_{ij}(t)=x_{-ir-js}(t),\hspace{5mm}
	h_{ij}(t)=h_{ir+js}(t),\hspace{5mm}
	n_{ij}=n_{ir+js}
\]
for $ir+js\in\Phi^+$ and $t\in K$. The root system $\Phi$ is depicted below.
% ls = {"01", "11", "32", "21", "31", "10"};
% \[Theta] = Pi/6;
% Export["hallpaige5.eps", Show[Graphics[Table[{
%           r = If[EvenQ[i], Sqrt[3], 1];
%           p = {Cos[i \[Theta]], Sin[i \[Theta]]};
%           Line[{-p r, p r}],
%           Text["\!\(x\_\(" <> ls[[i + 1]] <> "\)\)", p (r + .2)],
%           Text["\!\(y\_\(" <> ls[[i + 1]] <> "\)\)", -p (r + .2)]},
%         {i, 0, 5}]],
%     PlotRange -> {{-2.2, 2.2}, {-1.8, 1.8}},
%     AspectRatio -> Automatic,
%     TextStyle -> {FontSize -> 16, FontFamily -> Times},
%     DisplayFunction -> Identity]]
\[\includegraphics[width=60mm]{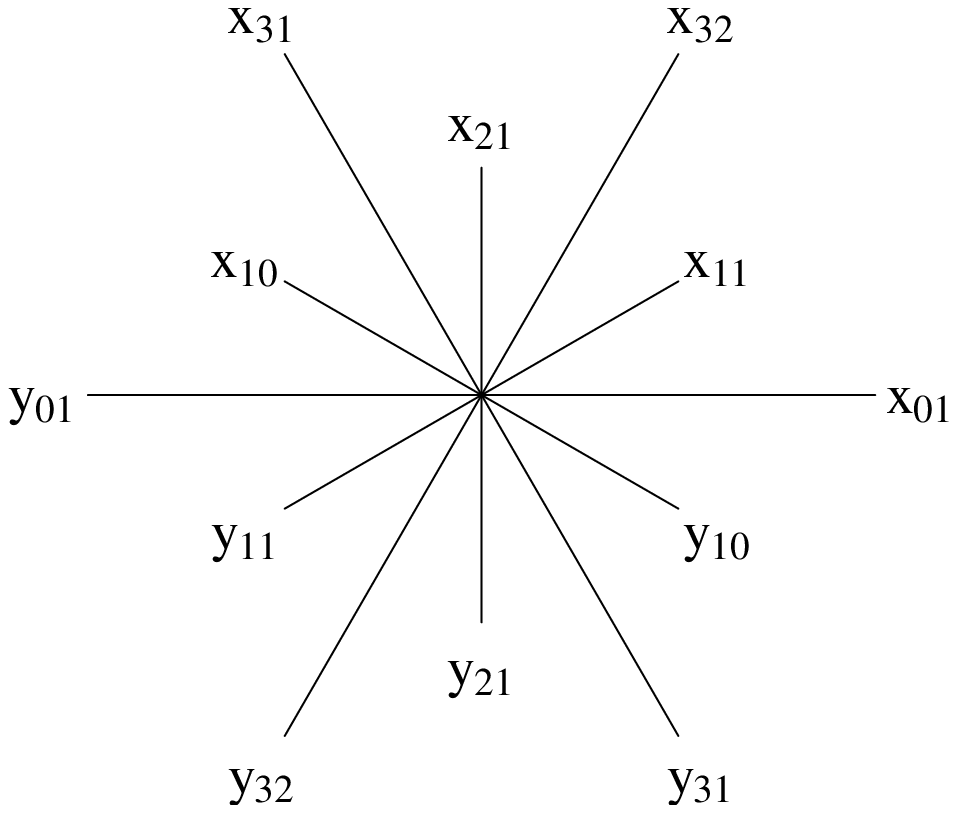}\]
Let $n=n_{11}n_{31}$. Then $n_{11}$ and $n_{31}$ commute by (\ref{Chev}), so
\[
	n^2=n_{11}^2n_{31}^2=h_{11}(-1)h_{31}(-1)=1,
\]
using (\ref{hr+s}). Also $n$ maps to the longest element of the Weyl group of $\bar G$, so $n\bar U n=\bar V$ (\cite{Carter} Lemma 7.2.1). 
We have an explicit description of the subgroups $U\subseteq\bU$ and $V\subseteq\bV$ (\cite{Carter} Propositions 13.6.1 and 13.6.3):
\begin{eqnarray}
	\hspace{-10mm}&&U=\{x_{10}(t^\theta)x_{01}(t)x_{11}(t^{\theta+1}+u^\theta)x_{21}(t^{2\theta+1}+v^\theta)x_{31}(u)x_{32}(v)
		\mid t,\,u,\,v\in K\},\label{U}\\
	\hspace{-10mm}&&V=\{y_{10}(t^\theta)y_{01}(t)y_{11}(t^{\theta+1}+u^\theta)y_{21}(t^{2\theta+1}+v^\theta)y_{31}(u)y_{32}(v)
		\mid t,\,u,\,v\in K\}.\label{V}
\end{eqnarray}
In particular,
\[
	n=x_{11}(1)x_{31}(1)y_{11}(-1)y_{31}(-1)x_{11}(1)x_{31}(1)
		\in UVU\subseteq G.
\]
Hence $n\in G\cap(\bar N-\bar H)=N-H$, and $nUn=G\cap n\bU n=G\cap\bV=V$, as required.
\end{proof}
\begin{lem}
There is a nontrivial element $a\in U$ such that $H\subseteq UVaV$.
\end{lem}
\begin{proof}
We give the calculation in the case ${}^2G_2(q)$. Let $a=x_{31}(1)x_{11}(1)$, which is in $U$ by (\ref{U}). We will show that any 
element $h\in H$ is in $UVaV$. By (\ref{hr+s}), we can write $h=h_{10}(\lambda)h_{01}(\mu)$ for some $\lambda,\,\mu\in K^*$. Recall that 
we have identified $\bar H$ with $\Hom(\Lambda,K^*)$ by defining
\[
	h_q(\lambda)(v)=\lambda^{\frac{2(q,v)}{(q,q)}}
\]
for $q\in\Phi$. By (\ref{Htwisted}),
\[
	\lambda^2\mu^{-1}=h(r)=h(s)^\theta=\lambda^{-3\theta}\mu^{2\theta}.
\]
Putting both sides to the power $3\theta-2$, and noting that $t^{3\theta^2}=t$ for $t\in K$, we obtain $\lambda=\mu^\theta$. Since 
$|K^*|=3^{2k+1}-1\equiv2\pmod{4}$, there are no elements of order $4$ in $K^*$, so $-1$ is not a square. Thus we can either write 
$\mu^{-1}=\nu^2$ or $\mu^{-1}=-\nu^2=\nu^2+\nu^2$. In either case $\mu^{-1}=\nu^2+\kappa^2$ for some $\nu\in K^*$ and $\kappa\in K$. 
Finally
\[
	(3\theta+1,q-1)=(3\theta+1,3\theta^2-1)=(3\theta+1,\theta+1)=(2,\theta+1)=2,
\]
so we can choose $\iota\in K^*$ with $\iota^{3\theta+1}=\nu^2$. We can therefore write
\begin{equation}\label{h}
	h=h_{10}(\iota^{\theta+1}+\kappa^{2\theta})^{-1}h_{01}(\iota^{3\theta+1}+\kappa^2)^{-1}
\end{equation}
with $\iota\neq0$. Given $t,u\in K$ with $tu\neq-1$, we have
\[
	\begin{pmatrix}1&0\\t&1\end{pmatrix}\begin{pmatrix}1&u\\0&1\end{pmatrix}
	=\begin{pmatrix}1&u\\t&1+tu\end{pmatrix}
	=\begin{pmatrix}1&u\lambda\\0&1\end{pmatrix}\begin{pmatrix}1&0\\t\lambda^{-1}&1\end{pmatrix}
		\begin{pmatrix}\lambda&0\\0&\lambda^{-1}\end{pmatrix},
\]
where $\lambda=(1+tu)^{-1}$. Applying the homomorphism $\rho_{ir+js}:SL_2(K)\rightarrow\bar G$ gives 
$y_{ij}(t)x_{ij}(u)=x_{ij}(u\lambda)y_{ij}(t\lambda^{-1})h_{ij}(\lambda)$. Using this and Chevalley's commutator 
formula (\ref{Chev}), and noting that $K$ has characteristic $3$, we calculate
\begin{eqnarray*}
	\bU y_{21}(t)y_{32}(u)x_{31}(v)x_{11}(w)\bV\hspace{-40mm}\\
		&=&\bU y_{21}(t)x_{31}(v)y_{01}(-uv)y_{32}(u)x_{11}(w)\bV\\
		&=&\bU x_{31}(v)x_{10}(-tv)y_{11}(t^2v)y_{32}(t^3v)y_{01}(t^3v^2)y_{21}(t)y_{01}(-uv)y_{32}(u)x_{11}(w)\bV\\
		&=&\bU y_{01}(t^3v^2-uv)y_{32}(t^3v+u)y_{11}(t^2v)y_{21}(t)x_{11}(w)\bV\\
		&=&\bU y_{01}(t^3v^2-uv)y_{32}(t^3v+u)y_{11}(t^2v)x_{11}(w)y_{10}(2tw)y_{21}(t)\bV\\
		&=&\bU y_{01}(t^3v^2-uv)y_{32}(t^3v+u)y_{11}(t^2v)x_{11}(w)\bV\\
		&=&\bU y_{01}(t^3v^2-uv)y_{32}(t^3v+u)x_{11}(w\mu)\bV h_{11}(\mu)
			\;\;\text{with }\mu=(1+t^2vw)^{-1}\\
% 		&=&\bU y_{01}(t^3v^2-uv)x_{11}(w\mu)x_{01}(-w^3\mu^3(t^3v+u))y_{10}(w^2\mu^2(t^3v+u))\\
% 		&&\;\;y_{21}(w\mu(t^3v+u))y_{32}(t^3v+u)y_{31}(w^3\mu^3(t^3v+u)^2)\bV h_{11}(\mu)\\
		&=&\bU y_{01}(t^3v^2-uv)x_{11}(w\mu)x_{01}(-w^3\mu^3(t^3v+u))\bV h_{11}(\mu)\\
% 		&=&\bU x_{11}(w\mu)x_{31}(-w^3\mu^3(t^3v^2-uv)^2)x_{32}(w^3\mu^3(t^3v^2-uv))x_{21}(-w^2\mu^2(t^3v^2-uv))\\
% 		&&\;\;x_{10}(w\mu(t^3v^2-uv))y_{01}(t^3v^2-uv)x_{01}(-w^3\mu^3(t^3v+u))\bV h_{11}(\mu)\\
		&=&\bU y_{01}(t^3v^2-uv)x_{01}(-w^3\mu^3(t^3v+u))\bV h_{11}(\mu)\\
		&=&\bU\bV h_{01}\left(1-vw^3\mu^3(t^3v+u)(t^3v-u)\right)^{-1}h_{11}(\mu)\\
		&=&\bU\bV h_{01}\left(1-vw^3\mu^3(t^6v^2-u^2)\right)^{-1}h_{01}(\mu)^3h_{10}(\mu)
			\;\;\text{by (\ref{hr+s})}\\
		&=&\bU\bV h_{01}\left((1+t^2vw)^3-vw^3(t^6v^2-u^2)\right)^{-1}h_{10}(\mu)\\
		&=&\bU\bV h_{01}\left(1+u^2vw^3\right)^{-1}h_{10}\left(1+t^2vw\right)^{-1},
\end{eqnarray*}
assuming that $1+t^2vw$ and $1+u^2vw^3$ are nonzero. Thus with $\iota$ and $\kappa$ as in (\ref{h}),
\begin{eqnarray*}
	\bU y_{21}(\kappa^\theta)y_{32}(\kappa)y_{31}(\iota-1)y_{11}(\iota^\theta-1)a\bV\hspace{-50mm}\\
		&=&\bU y_{21}(\kappa^\theta)y_{32}(\kappa)y_{31}(\iota-1)y_{11}(\iota^\theta-1)x_{31}(1)x_{11}(1)\bV\\
		&=&\bU y_{21}(\kappa^\theta)y_{32}(\kappa)x_{31}(\iota^{-1})x_{11}(\iota^{-\theta})\bV h_{11}(\iota^{-\theta})h_{31}(\iota^{-1})\\
		&=&\bU\bV h_{01}\left(1+\kappa^2\iota^{-3\theta-1}\right)^{-1}h_{10}\left(1+\kappa^{2\theta}\iota^{-\theta-1}\right)^{-1}
			h_{11}(\iota^{-\theta})h_{31}(\iota^{-1})\\
		&=&\bU\bV h_{10}\left(\iota^{\theta+1}+\kappa^{2\theta}\right)^{-1}h_{01}\left(\iota^{3\theta+1}+\kappa^2\right)^{-1}
			\;\;\text{by (\ref{hr+s})}\\
		&=&\bU\bV h.
\end{eqnarray*}
Now $y_{21}(\kappa^\theta)y_{32}(\kappa)y_{31}(\iota-1)y_{11}(\iota^\theta-1)\in V$ by (\ref{V}), so
\[
	h\in \bU Va\bV=\bU nUnan\bU n.
\]
Write $h=bncnandn$, where $b,d\in\bU$ and $c\in U$. Rearranging,
\[
	dnh^{-1}b=na^{-1}nc^{-1}n.
\]
Recall that $U$ and $V$ are the subgroups of $\bU$ and $\bV$ fixed by $\sigma$. The right hand side above is invariant under $\sigma$, 
so $dnh^{-1}b=\sigma(d)nh^{-1}\sigma(b)$. Now (\ref{normalform}) implies $d=\sigma(d)$ and $b=\sigma(b)$; that is, $d,\,b\in U$. Hence 
$h\in UVaV$, as required.

In the notation of the previous proof, we may take
\[
	a=\begin{pmatrix}1&1\\ 0&1\end{pmatrix}
\]
in the $A_1$ case, and
\[
	a=\begin{pmatrix}1&0&\epsilon\\ 0&1&0\\ 0&0&1\end{pmatrix}
\]
in the ${}^2A_2$ case, where $\epsilon\in K^*$ satisfies $\epsilon+\bar\epsilon=0$. We omit these calculations as 
they are similar to, but much easier than, the ${}^2G_2$ calculation.
\end{proof}
\begin{lem}\label{UnhU}
The double cosets of $U$ satisfy $(UnhU)(Unh'U)\supseteq(Unh''U)$ for all $h,h',h''\in H$.
\end{lem}
\begin{proof}
Let $a\in U$ be the nontrivial element constructed in the previous lemma. Suppose $nan\in B$. Then 
$nan=hb$ for some $h\in H$ and $b\in U$, giving $an=nhb$. Expressions of the form $UnHU$ are unique, so $a=1$, 
a contradiction. Hence $nan\notin B$, so $nan\in UnhU$ for some $h\in H$. Since $H$ normalizes $U$, we obtain
\[
	H\subseteq UnUnanUn
		\subseteq UnU(UnhU)Un
		=UnUnUnh^n.
\]
But $n$ normalises $H$, so $h^n\in H$. Hence $H\subseteq UnUnUn$. Now for arbitrary $h,h',h''\in H$, we have
\[
	nh''\in nHh^nh'=Hnh^nh'\subseteq UnUnUh^nh'=UnhUnh'U.
\]
Thus $(UnhU)(Unh'U)\supseteq Unh''U$, as required.
\end{proof}
This lemma suggests that we should apply Corollary \ref{dcstperm} to the subgroup $U$; indeed $|U|$ is odd by (\ref{|U|}), so $U$ 
possesses a complete mapping. However, the normaliser of $U$ is $B$, and $B/U\cong H$. Since $l=1$, the group $\hat{H}$ is cyclic, 
so $H$ is a cyclic group of even order. This implies that no permutations of $\HGH{U}$ can satisfy the conditions of Corollary 
\ref{dcstperm}. Nevertheless we can come close using the following lemma, which says that $H$ falls one equation short of having a 
complete mapping.
\begin{lem}\label{cyc}
If $C$ is a cyclic group of even order, then there exist permutations $\alpha$ and $\beta$ of $C$ such that $c\alpha(c)=\beta(c)$ 
for $c\neq1$. Moreover we can take $\beta(1)=1$ and $\alpha(1)\neq1$.
\end{lem}
\begin{proof}
Identify $C$ with $\Z_{2k}$, written additively. Let $\alpha(0)=k$. For $1\leq i<k$, let $\alpha(i)=i$, and for $k\leq i<2k$, let 
$\alpha(i)=i+1$. For $0\leq i<k$, let $\beta(i)=2i$, and for $k\leq i<2k$, let $\beta(i)=2i+1$. It is clear that $i+\alpha(i)=\beta(i)$ 
for $i\neq0$. Also $\beta(i)$ takes all the even values for $0\leq i<k$, and all the odd values for $k\leq i<2k$. Thus $\beta$ is a 
permutation, and it is easy to see that $\alpha$ is a permutation also.
\end{proof}
The proof of Corollary \ref{dcstperm} constructs permutations of the left cosets satisfying the conditions of Proposition \ref{lcst}. 
To prove the next result, we apply the same construction to permutations of $\HGH{U}$ which don't quite satisfy the required conditions. 
After some tweaking, we can apply Proposition \ref{lcst} directly.
\begin{lem}\label{smallliegroups}
Suppose $G$ has type $A_1$, ${}^2A_2$ or ${}^2G_2$, $q$ is odd and $|H|$ is even. Then $G$ possesses a complete mapping.
\end{lem}
\begin{proof}
The left cosets of $U$ in $G$ are exactly $hU$ and $unhU$ for $u\in U$ and $h\in H$. Let $I=H\amalg(U\times H)$. We will define bijections 
$\bar x$, $\bar y$ and $\bar z$ from $I$ to $G/U$ as follows. Firstly, for any $h\in H$, Lemma \ref{UnhU} gives
\[
	nhU\subseteq UnhUnhU.
\]
Thus there exists $v_h\in U$ such that $nhU\subseteq v_hnhUnhU$. Let $\alpha$ and $\beta$ be the permutations of $H$
given by Lemma \ref{cyc}. Define
\[\begin{array}{*{2}{r@{\;=\;}l@{\hspace{10mm}}}r@{\;=\;}l}
	\bar x(h)&hU&\bar y(h)&\alpha(h)U&\bar z(h)&\beta(h)U\\
	\bar x(u,h)&uv_hnhU&\bar y(u,h)&unhU&\bar z(u,h)&unhU
\end{array}\]
Since $\alpha$ and $\beta$ are permutations, $\bar x$, $\bar y$ and $\bar z$ are certainly bijections from $I$ to $G/U$. Also
\[
	\bar x(u,h)\bar y(u,h) =uv_hnhUunhU=u(v_hnhUnhU)\supseteq unhU=\bar z(u,h)
\]
for any $(u,h)\in U\times H$, and
\[
	\bar x(h)\bar y(h)=h\alpha(h)U=\beta(h)U=\bar z(h),
\]
provided $h\neq1$. Unfortunately this does not hold when $h=1$. We therefore tweak a few values; put $\zeta=\alpha(1)^{-1}\neq 1$, and define
\[\begin{array}{*{2}{r@{\;=\;}l@{\hspace{10mm}}}r@{\;=\;}l}
	\T x(1)&v_\zeta v_1nU&\T y(1)&v_\zeta nU&\T z(1)&U\\
	\T x(v_\zeta,1)&v_\zeta n\zeta U&\T y(v_\zeta,1)&\alpha(1)U&\T z(v_\zeta,1)&v_\zeta nU\\
	\T x(1,\zeta)&U&\T y(1,\zeta)&n\zeta U&\T z(1,\zeta)&n\zeta U.
\end{array}\]
Define $\T x$, $\T y$ and $\T z$ to coincide with $\bar x$, $\bar y$ and $\bar z$ on $I-\{1,(v_\zeta,1),(1,\zeta)\}$. Recalling that $\beta(1)=1$, 
we have
\begin{eqnarray*}
	\bar x(\{1,(v_\zeta,1),(1,\zeta)\})
		&=&\{U,v_\zeta v_1nU,v_\zeta n\zeta U\}
		\;=\;\T x(\{1,(v_\zeta,1),(1,\zeta)\}),\\
	\bar y(\{1,(v_\zeta,1),(1,\zeta)\})
		&=&\{\alpha(1)U,v_\zeta nU,n\zeta U\}
		\;=\;\T y(\{1,(v_\zeta,1),(1,\zeta)\}),\\
	\bar z(\{1,(v_\zeta,1),(1,\zeta)\})
		&=&\{U,v_\zeta nU,n\zeta U\}
		\;=\;\T z(\{1,(v_\zeta,1),(1,\zeta)\}).
\end{eqnarray*}
Thus $\T x$, $\T y$ and $\T z$ are also bijections. Also
\begin{eqnarray*}
	\T x(1)\T y(1)
		&=&v_\zeta v_1nUv_\zeta nU
		\;\supseteq\;v_\zeta v_1n^2U
		\;=\;U
		\;=\;\T z(1),\\
	\T x(v_\zeta,1)\T y(v_\zeta,1)
		&=&v_\zeta n\zeta U\alpha(1)U
		\;=\;v_\zeta n\alpha(1)^{-1}\alpha(1)U
		\;=\;v_\zeta nU
		\;=\;\T z(v_\zeta,1),\\
	\T x(1,\zeta)\T y(1,\zeta)
		&=&Un\zeta U
		\;\supseteq\;n\zeta U
		\;=\;\T z(1,\zeta).
\end{eqnarray*}
Thus $\T x(i)\T y(i)\supseteq\T z(i)$ for all $i\in I$, so Proposition \ref{lcst} gives the result.
\end{proof}
Summarising these results, we have:
\begin{thm}
Suppose $G$ is a minimal counterexample to the HP conjecture. Then $G$ is one of the 26 sporadic simple groups or the Tits group.
\end{thm}
\begin{proof}
By Theorem \ref{simple}, the group $G$ must be simple. Therefore $G$ is either a cyclic group, an alternating group, a simple group of Lie 
type, the Tits group, or one of the 26 sporadic groups \cite{atlas}. The HP conjecture holds for cyclic groups by Proposition \ref{solu}. 
It holds for alternating groups by Theorem 3 of \cite{HP}. Suppose $G$ is a simple group of Lie type. Suppose $G$ is not covered by Lemma 
\ref{smallliegroups}. Then Lemmas \ref{Pgood} and \ref{P^2} show that $G$ has a good proper subgroup $P$ whose double cosets $D$ satisfy 
$D^2\supseteq D$. By the minimality assumption, $P$ admits a complete mapping, so Corollary \ref{dcst} shows that $G$ admits a complete 
mapping. Therefore the only remaining groups are the sporadic groups and the Tits group.
\end{proof}
\section{acknowledgements}
The author would like to thank Peter McNamara and Martin Isaacs for useful discussions concerning Lemma \ref{p}, and also Bob Howlett for explaining how to find minimal double coset representatives with MAGMA.
\bibliographystyle{plain}
% \bibliography{hallpaige}

\end{document}